\documentclass[oneside,english]{article}
\usepackage[T1]{fontenc}
\usepackage{amsthm}
\usepackage{amssymb}
\usepackage[all]{xy}
\usepackage{amsmath}
\usepackage{comment}

\usepackage{color}
\usepackage{mathtools}
\usepackage{tikz-cd}
\tikzcdset{scale cd/.style={every label/.append style={scale=#1},
		cells={nodes={scale=#1}}}}
\usepackage{soul}
\usepackage{xcolor}
\usepackage[many]{tcolorbox}
\tcbset{
	highlight math style={
		colback=yellow,
		arc=0pt,
		outer arc=0pt,
		boxrule=0pt,
		top=2pt,
		bottom=2pt,
		left=2pt,
		right=2pt,
	}
}
\usepackage{framed}
\definecolor{shadecolor}{RGB}{241,231,64}

\makeatletter
\newcommand{\nc}{\newcommand}
\newtheorem{thm}{Theorem}
\newtheorem*{thm*}{Theorem}
\theoremstyle{plain}
\nc{\bthm}{\begin{thm}} \nc{\ethm}{\end{thm}}
\newtheorem{prop}[thm]{Proposition}
\nc{\bprp}{\begin{prop}} \nc{\eprp}{\end{prop}}
\newtheorem{fact}[thm]{Fact}
\nc{\bfct}{\begin{fact}} \nc{\efct}{\end{fact}}
\newtheorem{prob}[thm]{Problem}
\nc{\bprb}{\begin{prob}} \nc{\eprb}{\end{prob}}
\newtheorem{lem}[thm]{Lemma}
\nc{\blem}{\begin{lem}} \nc{\elem}{\end{lem}}
\newtheorem{claim}[thm]{Claim}
\nc{\bclm}{\begin{claim}} \nc{\eclm}{\end{claim}}
\newtheorem{cor}[thm]{Corollary}
\nc{\bcor}{\begin{cor}} \nc{\ecor}{\end{cor}}
\newtheorem{conj}[thm]{Conjecture}
\nc{\bcnj}{\begin{conj}} \nc{\ecnj}{\end{conj}}
\theoremstyle{definition}
\newtheorem{defn}[thm]{Definition}
\nc{\bdfn}{\begin{defn}} \nc{\edfn}{\end{defn}}
\newtheorem{observation}[thm]{Observation}
\nc{\bobs}{\begin{observation}} \nc{\eobs}{\end{observation}}
\theoremstyle{remark}
\newtheorem{rem}[thm]{Remark}
\nc{\brem}{\begin{rem}} \nc{\erem}{\end{rem}}
\newtheorem{cnv}[thm]{Convention}
\nc{\bcnv}{\begin{cnv}} \nc{\ecnv}{\end{cnv}}
\newtheorem{exam}[thm]{Example}
\nc{\bexm}{\begin{exam}} \nc{\eexm}{\end{exam}}

\nc{\bpf}{\begin{proof}} \nc{\epf}{\end{proof}}
\nc{\be}{\begin{enumerate}}
	\nc{\ee}{\end{enumerate}}
\nc{\bi}{\begin{itemize}}
	\nc{\itm}{\item}
	\nc{\ei}{\end{itemize}}
\nc{\invlim}{\lim_{\leftarrow}}
\nc{\dirlim}{\lim_{\rightarrow}}
\nc{\mm}{\mathbf{m}}
\nc{\nn}{\mathbf{n}}
\nc{\FF}{\mathcal{F}}
\nc{\CC}{\mathcal{C}}
\nc{\Span}{\operatorname{span}}
\nc{\Img}{\operatorname{Im}}
\nc{\rank}{\operatorname{rank}}
\nc{\proj}{\operatorname{proj}}
\nc{\F}{\mathbb{F}}
\nc{\Z}{\mathbb{Z}}
\nc{\Q}{\mathbb{Q}}
\nc{\Br}{\operatorname{Br}}

\begin{document}
	\title{Kurosh Subgroup Theorem in Free Pro-$p$ Products}
	\author{Tamar Bar-On}
	\date{}
	\maketitle 
	\begin{abstract}
We provide a sufficient condition under which a closed subgroup of a restricted free pro-$p$
 product is itself a free pro-$p$ product.
	\end{abstract}
\section*{Introduction}
We recall the famous \textit{Kurosh subgroup Theorem} for subgroups of abstract free products:
\begin{thm}[{\cite[Theorem I.14]{serre1980sl}}]
	Let $G=G_1\ast G_2\ast\cdots \ast G_n$ be the abstract free product of the abstract groups $G_1,...,G_n$, and let $H$ be a subgroup of $G$. Then $$H=[\ast_{i=1}^n\ast_{\tau\in H\backslash G/G_i}(H\cap g_{i,\tau}G_ig_{i,\tau}^{-1}) ]\ast F$$ where for each $i$, $g_{i,\tau}$ runs over a system of double coset representatives for $H\backslash G/G_i$ which contains 1, and $F$ is an abstract free group.  
\end{thm}
The Kurosh subgroup Theorem admits a weaker version for pro-$\mathcal{C}$ groups and free pro-$\mathcal{C}$ products. More precisely:
\begin{thm}[{\cite[Theorem 9.1.9]{ribes2000profinite}}]
	Let $\mathcal{C}$ be an extension closed variety of finite groups, $G=G_1\amalg^{\mathcal{C}} G_2\amalg^{\mathcal{C}}\cdots \amalg^{\mathcal{C}} G_n$ be the free pro-$\mathcal{C}$ product of the pro-$\mathcal{C}$ groups $G_1,...,G_n$, and let $U$ be an \textbf{open} subgroup of $G$. Then $$U=\left[\coprod_{i\in\{1,...,n\}}^{\mathcal{C}}\coprod_{\tau\in H\backslash G/G_i}^{\mathcal{C}}(H\cap g_{i,\tau}G_ig_{i,\tau}^{-1}) \right]\amalg^{\mathcal{C}} F$$ where for each $i$, $g_{i,\tau}$ runs over a system of double coset representatives for $U\backslash G/G_i$ which contains 1, and $F$ is a free pro-$\mathcal{C}$ group.  
\end{thm}
Recall that an extension-closed variety of finite groups is a class of finite groups which is closed under taking subgroups, quotients and extensions.

The above version was extended to \textit{restricted} free pro-$\mathcal{C}$ products (to be defined later) of infinitely many groups. More precisely:
\begin{thm}[{\cite[Corollary 7.3.2]{ribes2017profinite}}]\label{open subgroup of restricted}
	Let $\mathcal{C}$ be an extension-closed variety of finite groups, $G=\coprod_{i\in I}^{r,\mathcal{C}} G_i$ be the restricted free pro-$\mathcal{C}$ product of the pro-$\mathcal{C}$ groups $\{G_i, i\in I\}$, and let $U$ be an \textbf{open} subgroup of $G$. Then there exist sets $D_i$ of double coset representatives for $U\backslash G/G_i$ which contain 1, such that the family of subgroups $$\{U\cap xG_ix^{-1}|i\in I,x\in D_i\}$$ converges to 1, and $U$ is the restricted free pro-$\mathcal{C}$ product $$U=\left[\coprod_{i\in I,x\in D_i}^{r,\mathcal{C}}(U\cap xG_ix^{-1}) \right]\amalg^{\mathcal{C}} F$$ where $F$ is a free pro-$\mathcal{C}$ group.  
\end{thm}
Unfortunately, the general version of the Kurosh subgroup Theorem does not hold for free pro-$\mathcal{C}$ products. In \cite[Section 10.7]{ribes2017profinite} the author gave an example of two pro-$p$ groups $G_1, G_2$ and a closed subgroup $H$ of the free pro-$p$ product $G_1\coprod^p G_2$, such that $H$ does not admit a decomposition as a free pro-$p$ product of a free pro-$p$ group and the intersections
of $H$ with conjugates of  $G_1$ and $G_2$.

However, some special cases are known. In \cite[Theorem 9.6.2]{ribes2017profinite} it has been proven that if $\mathcal{C}$ is an extension-closed variety of finite groups which contains $C_p$ for some prime $p$, and $H$ is a second-countable pro-$p$ subgroup of a free pro-$\mathcal{C}$ product $\coprod_T^{\mathcal{C}}G_t$ over a profinite space $T$ (to be defined later), then $H$ admits a decomposition as a free pro-$p$ product of a free pro-$p$ group and intersections
of $H$ with conjugates of the subgroups $G_t$, for all $t\in T$. Over the years, significant effort has been devoted to the study of closed subgroups of free pro-$\mathcal{C}$ products, and mainly for those satisfying the Kurosh Subgroup Theorem, as can be shown, for example, in \cite{haran1987closed,herfort1987solvable,herfort1987subgroups} and \cite{jarden1994prosolvable}.  

The object of this paper is to prove the following:
\begin{thm*}[Main Theorem]
	Let $G=\coprod_{i\in I}^{r,p}G_i$ be a restricted free pro-$p$ product and $H\leq G$ a closed subgroup. Assume that:
	\begin{enumerate}
		\item The space $H\backslash G/G_i$ is finite for every $i\in I$, and
		\item There exist sets $D_i$ of double coset representatives for $H\backslash G/G_i$, such that $H$ is generated by the subgroups $\{H\cap xG_ix^{-1}:i\in I,x\in D_i\}$.
	\end{enumerate} 
Then the set $\{H\cap xG_ix^{-1}:i\in I,x\in D_i\}$ of subgroups of $H$ converges to 1, and  $H=\coprod_{i\in I,x\in D_i}^{r,p}(H\cap xG_ix^{-1})$.
	\end{thm*}
It is worth mentioning that the proof given in this paper is purely group-theoretic and elementary, while the proof of \cite[Theorem 9.6.2]{ribes2017profinite} relies on the theory of profinite groups acting on profinite trees.

\section*{Main results}
Although we focus on restricted free pro-$p$ products in this paper, some of the results can be done generally for free pro-$\mathcal{C}$ products of families of subgroups continuously indexed by profinite spaces. Hence, we start this section by recalling the notation of a free pro-$\mathcal{C}$ product of a family of subgroups continuously indexed by a profinite space. First, we present the internal point of view. For an extensive discussion on free pro-$\mathcal{C}$ products, from external and internal points of view, the reader is referred to \cite[Chapter 5]{ribes2017profinite}.

Throughout this paper, $\mathcal{C}$ is always assumed to be an extension-closed variety of finite groups.

Let $T$ be a profinite space, $G$ a pro-$\mathcal{C}$ group, and let $\{G_t:t\in T\}$ be a family of closed subgroups of $G$ indexed by $T$. We say that the family $\{G_t:t\in T\}$ is \textit{continuously} indexed by $T$ if for every open subgroup $U\leq_o G$, the set $\{t\in T:G_t\leq U\}$ is open in $T$.
\begin{rem}[{\cite[Lemma 5.2.1]{ribes2017profinite}}]\label{continuously indexed}
	The property of $\{G_t:t\in T\}$ being continuously indexed by $T$ is equivalent to the subspace $\mathcal{G}=\{(g,t):g\in G_t\}$ being closed in $G\times T$.
\end{rem}
$G$ is said to be the free pro-$\mathcal{C}$ product of $\{G_t:t\in T\}$, and we denote $G=\coprod_T^{\mathcal{C}}G_t$, if the following properties are satisfied:
\begin{enumerate}
	\item $\{G_t:t\in T\}$ is continuously indexed by $T$.
	\item For every $t\ne s\in T$, $G_t\cap G_s=\{e\}$.
	\item $G=\overline{\langle G_t:t\in T\rangle }$.
	\item Let $D=\bigcup_TG_t$. Then every continuous function $f:D\to H$, where $H$ is a pro-$\mathcal{C}$ group, such that $f|_{G_t}$ is a homomorphism for every $t\in T$, can be lifted (uniquely) to a continuous homomorphism $G\to H$.
\end{enumerate}
When $\mathcal{C}$ is the variety of all finite $p$-groups we write $G=\coprod_T^{p}G_t$ instead of $G=\coprod_T^{\mathcal{C}}G_t$.

In case $T=I\cup \{*\}$ is the one-point compactification of the discrete space $I$, and $G_{*}=\{e\}$, then $G$ is called a \textit{restricted} free pro-$\mathcal{C}$ product of $\{G_i\}_{i\in I}$, and is denoted by $G=\coprod_T^{r,\mathcal{C}}G_t$. In that case the property of $\{G_t:t\in T\}$ being continuously indexed by $T$ is equivalent to the set $\{G_i:i\in I\}$ being \textit{converging to 1}, i.e, to the property that every open subgroup $U\leq_o G$ contains all but finitely many $G_i$'s. In addition, a map $\varphi:D\to H$ which is a continuous homomorphism on every $G_t$ is continuous on $D$ if and only if it converges to 1, meaning that every open subgroup $U\leq_oH$ contains all but finitely many $\varphi(G_i)$'s.

We briefly recall the notation of an \textit{external} free pro-$\mathcal{C}$ product. Let ${\pi:\mathcal{G}\to T}$ be a quotient map between two profinite spaces such that for every $t\in T$, $\pi^{-1}(t)$ is a pro-$\mathcal{C}$ group with respect to the induced topology. Assume further that the map $(x,y)\to xy^{-1}$ is continuous as a map from $\mathcal{G}^2=\{(x,y):\pi(x)=\pi(y)\}$ to $\mathcal{G}$ where the former space is equipped with the induced product topology. Then $(\mathcal{G},T,\pi)$ is called a \textit{sheaf} of pro-$\mathcal{C}$ groups. The free pro-$\mathcal{C}$ product over $(\mathcal{G},T,\pi)$, denoted by $\coprod^{\mathcal{C}}_T\mathcal{G}$, is a pro-$\mathcal{C}$ group $G$, equipped with a continuous map $\varphi:\mathcal{G}\to G$ which is a homomorphism on every fiber $\pi^{-1}(t)$, which satisfies the following universal property: for every continuous map into a pro-$\mathcal{C}$ group $f:\mathcal{G}\to H$, which is a homomorphism on each fiber, there exists a unique continuous homomorphism $\tilde{f}:G\to H$ such that $\tilde{f}\circ \varphi=f$.

By \cite[Section 5.3]{ribes2017profinite}, every external free pro-$\mathcal{C}$ product $\coprod_T^{\mathcal{C}}\mathcal{G}$ is an internal free pro-$\mathcal{C}$ product of its family of subgroups $\varphi(\pi^{-1}(t))$, and conversely, every internal free pro-$\mathcal{C}$ product of a family of subgroups $\{G_t:t\in T\}$ is an external free pro-$\mathcal{C}$ product of the sheaf $(\mathcal{G},T,\pi)$, that was defined in Remark \ref{continuously indexed}, where $\pi(g,t)=t$.

Our first tool in proving the main theorem is the following condition for being a free pro-$\mathcal{C}$ product, which becomes an equivalent condition for restricted free pro-$\mathcal{C}$ products. This is a general version of the characterization of free pro-$\mathcal{C}$ products of finitely many groups that was presented in \cite{haran2000free}.

First, we need the following notations: Let $G$ be a pro-$\mathcal{C}$ group and ${\{G_t:t\in T\}}$ a family of closed subgroups of $G$. An embedding problem for ${(G,\{G_t:t\in T\})}$ in the category of pro-$\mathcal{C}$ groups is a tuple 
\begin{equation}\label{embedding problem}
	\begin{gathered}
	 (\varphi:G\to A,\alpha:B\to A,\{B_t:t\in T\})
	\end{gathered}
\end{equation}
consisting of two continuous homomorphisms and a set of subgroups of $B$, such that $B,A$ are pro-$\mathcal{C}$ groups, $\alpha$ is an epimorphism, and for every $t\in T$,  $\alpha$ maps $B_t$ isomorphically onto $\varphi(G_t)$. The embedding problem (\ref{embedding problem}) is called \textit{finite} if $B$ is finite. The embedding problem (\ref{embedding problem}) is called \textit{solvable} if there exists a homomorphism $\psi:G\to B$ such that $\alpha\circ\psi=\varphi$, and for every $t\in T$, $\psi(G_t)=B_t$.  
\begin{rem}
	One should not be confused with the definition of a solvable embedding problem relative to a family of closed subgroups which appears in the context of relatively projective groups and has a different meaning (see, for example, \cite{haran2005p,koenigsmann2002relatively}).
\end{rem}
\begin{prop}\label{condition for being free product}
	Let $G$ be a pro-$\mathcal{C}$ group and $\{G_t:t\in T\}$ a family of closed subgroups continuously indexed by $T$ such that $G=\overline{\langle G_t:t\in T\rangle }$. Then $G$ is a free pro-$\mathcal{C}$ product of $\{G_t:t\in T\}$ if every finite embedding problem for $(G,\{G_t:t\in T\})$ in the category of pro-$\mathcal{C}$ groups is solvable.
\end{prop}
\begin{proof}
	Let $H$ be the free pro-$\mathcal{C}$ product over the sheaf $(\mathcal{G},T,\pi)$, as defined in Remark \ref{continuously indexed}. Observe that by  Remark \ref{continuously indexed}, $\mathcal{G}$ is indeed a profinite space. Moreover, one immediately checks that the map $\mathcal{G}^2\to \mathcal{G}$, which is defined by $((x,t),(y,t))\to (xy^{-1},t)$, is continuous. In particular, $(\mathcal{G},T,\pi)$ is indeed a sheaf of pro-$\mathcal{C}$ groups. For more convenience denote by $H_t$ the image of $\pi^{-1}(t)=G_t\times \{t\}$ inside $H$. In particular, the canonical map from $\mathcal{G}$ to $H$ maps $\pi^{-1}(t)$ isomorphically onto $H_t$ (see \cite[Proposition 5.1.6 (c)]{ribes2017profinite}). Let $\varphi:H\to G$ be the map induced by the continuous function $\mathcal{G}\to G$ which is defined by $(g,t)\to g$. Observe that $\varphi$ maps $H_t$ isomorphically onto $G_t$. Since $G$ is generated by $\{G_t:t\in T\}$ then $\varphi$ is in fact onto. It is enough to construct a group-theoretic continuous section $\psi:G\to H$ such that for every $t\in T$ $\psi(G_t)=H_t$. Indeed, assume that such a section $\psi$ exists. Since $\varphi\circ \psi=Id_G$ then $\psi$ is injective. On the other hand, since $H$ is generated by $\{H_t:t\in T\}$ then $\psi$ is surjective. We conclude that $\psi$ is in fact an isomorphim.
	
	We will prove a more general claim: Let $H$ be a pro-$\mathcal{C}$ group and let $\{H_t,t\in T\}$ be a family of closed subgroups continuously indexed by $T$. Let $\varphi:H\to G$ be a continuous epimorphism which maps $H_t$ isomorphically onto $G_t$ for every $t\in T$. Then $\varphi$ admits a continuous group-theoretic section $\psi:G\to H$ such that $\psi(G_t)=H_t$ for every $t\in T$.
	
	Let $K=\ker(\varphi)$. Clearly $H_t\cap K=\{e\}$. For every closed normal subgroup $N\unlhd H$ such that $N\leq K$ let $\varphi_N:H/N\to G$ be the induced projection. We claim that there exists a closed normal subgroup $N\leq K$ such that the projection $\varphi_N:H/N\to G$ admits a continuous section $\psi_N:G\to H/N$ which satisfies $\psi_N(G_t)=H_tN/N$ for every $t\in T$, which is minimal in the following sense:
	 
	 let us define the following poset: let $\mathcal{A}=\{(M,\psi_M)\}$ be the set of all pairs consisting of a closed normal subgroup of $G$ which is contained in $K$ and a continuous section $\psi_M:G\to H/M$ such that $\psi_M(G_t)=H_tM/M$ for every $t\in T$. $\mathcal{A}\ne \emptyset$ as it contains $(H/K,\varphi_K^{-1})$. We define a partial order on $\mathcal{A}$ by letting $(M,\psi_M)\geq(L,\psi_L)$ if and only if $M\leq L$ and $\psi_L=\pi_{ML}\circ \psi_M$, where $\pi_{ML}:H/M\to H/L$ is the natural projection. By Zorn's lemma it is enough to show that every chain in $\mathcal{A}$ admits an upper bound. So, let $\{(M_i,\psi_{M_i})\}$ be a chain in $\mathcal{A}$. Set $L=\bigcap M_i$. Clearly $L\unlhd G$ and $L\leq K$. In addition, $H/L\cong\varprojlim (H/M_i,\pi_{M_i,M_j})$. Hence the compatible sections $\varphi_{M_i}$ induce a continuous homomorphsim $\psi_L:G\to H/L$. One checks immediately that $\psi_L$ is a section for $\varphi_L$ and $\psi_L(G_t)=\bigcap_i H_tM_i=H_t\bigcap_i M_i=H_tL$ for all $t\in T$ (see \cite[Proposition 2.1.4 (a)]{ribes2000profinite} for the second equation). Hence, we can choose a maximal element $(N,\psi_N)\in \mathcal{A}$. Such an element is in fact \textit{minimal} in the sense defined above.
	If $N=\{e\}$ we are done. Assume by contradiction that $N\ne \{e\}$. Hence $N$ admits a proper open normal subgroup $M\unlhd_o N$. By basic properties of profinite groups (see \cite[Lemma 1.2.5 (a)]{fried2005field}), there exists an open normal subgroup $L\unlhd_oH$ such that $L\cap N\leq M$. Hence we may replace $M$ by $L\cap N$ and assume that $M\unlhd H$. Let us look at the following embedding problem:
	\begin{equation}\label{finite kernel}
		\begin{gathered}
		\xymatrix@R=14pt{ &&&G\ar[d]^{\psi_N}  \\
			1\ar[r] & N/M\ar[r] &  H/M\ar[r]^{\pi_{MN}} & H/N\ar[r]&1}
		\end{gathered}
	\end{equation}
	
Observe that for every $t\in T$, $H_tM/M$ is mapped isomorphically onto $H_tN/N$. It is enough to show that    \eqref{finite kernel} is solvable. Recall again that for every $t\in T$ $H_t\cap K=\{e\}$. In particular, $H_t\cap N\leq M$. Hence there exists some $U_t\leq _oG$ such that $U_tH_t\cap N\leq M$. In particular, if $s\in T$ satisfies that $H_s\leq U_tH_t$ then $U_tH_s\cap N\leq M$. Since $\{H_t:t\in T\}$ is continuously indexed by $T$, the set $O_t=\{s\in T:H_s\leq U_tH_t\}$ is open in $T$. In other words, the set $\{O_t:t\in T\}$ is an open cover for $T$ and hence admits a finite sub-cover. Let $O_{t_1},...,O_{t_n}$ be such a sub-cover and define $U$ to be an open normal subgroup of $G$ which is contained in $U_{t_1},...,U_{t_n}$, such that $U\cap N\leq M$. Then $H_tU\cap N\leq M$ for all $t\in T$. Let us look at the following diagram:
\begin{equation*}
	\begin{gathered}
		\xymatrix@R=14pt{ &&&G\ar[d]^{\psi_N}  \\
			1\ar[r] & N/M\ar[r] &  H/M\ar[r]^{\pi_{MN}} \ar[d]_{\pi_{M,MU}} & H/N\ar[r] \ar[d]^{\pi_{N,NU}}&1\\
		&&H/MU\ar[r]_{\pi_{MU,NU}}& H/NU\ar[r]&1}
	\end{gathered}
\end{equation*}
\sloppy
$H/MU$ is a finite pro-$\mathcal{C}$ group. In addition $\pi_{N,NU}\circ\psi_N$ sends $G_t$ to $H_tNU/NU$ for every $t\in T$, and $\pi_{MU,NU}$ maps $H_tMU$ isomorphically onto $H_tNU$ since $H_tMU\cap NU\leq MU$ for every $t\in T$. Hence by assumption there exists a solution $\eta:G\to H/MU$ such that $\eta(G_t)=H_tMU/MU$ for every $t\in T$. Since $MU\cap N=M$ and $MUN=NU$, $(H/M,\pi_{MN},\pi_{M,MU})$ is the fiber product (also known as \textit{pullback}) of \begin{equation*}
	\begin{gathered}
		\xymatrix@R=14pt{
		 &   H/N\ar[d]^{\pi_{N,NU}}\\
			H/MU\ar[r]_{\pi_{MU,NU}}& H/NU}
	\end{gathered}
\end{equation*} (see \cite[Exercise 2.10.1]{ribes2000profinite}). The identification of $H/M$ with $H/MU\times_{H/NU}H/N$ is done by the map $hM\to (hMU,hN)$. By definition of the fiber product there exists a continuous homomorphism $\psi_M:G\to H/M$ such that $\pi_{MN}\circ\psi_M=\psi_N$. Since $\varphi_M=\varphi_N\circ\pi_{MN}$ and $\varphi_N\circ\psi_N=Id_G$ one concludes that $\psi_M$ is indeed a section for $\varphi_M$. In addition, let $g\in G_t$. Then $\psi_M(g)=hM$ for the unique $hM\in H/M$ such that $(hMU,hN)=(\eta(g),\psi_N(g))$. In particular, $hM\leq H_tMU\cap H_tN=H_tM$. Hence, $\psi(G_t)\leq H_tM/M$ for every $t\in T$. However, since $H_tM/M$ maps isomorphically onto $G_t$, $\psi(G_t)=H_tM/M$. We contradicted the minimality of $(N,\psi_N)$, hence we are done.
\end{proof}

\begin{rem}
	\begin{enumerate}
		\item As can be shown by the proof of Proposition \ref{condition for being free product}, it is enough to require $G$ to solve embedding problems for which $\{B_t:t\in T\}$ is continuously indexed by $T$.
		\item As can be shown by the proof of Proposition \ref{condition for being free product}, it is enough to require $G$ to finite solve embedding problems for which $\ker(\alpha)$ is a minimal normal subgroup. Indeed, Let $N$ be as in the proof. If $N=\{e\}$ we are done. Otherwise, as was explained in the proof, $N$ contains a proper open subgroup $M$ which is normal in $G$. Let $M$ be maximal with respect to these two properties. Now let $U$ be as in the proof. In particular, $MU\cap N=M$. We want to show that the natural projection $H/MU\to H/NU$ admits a minimal normal kernel. In other words we shall show that $NU/MU$ is a minimal normal subgroup in $H/MU$, or equivalently, that there is no normal subgroup $L\unlhd H$ which lies properly between $MU$ and $NU$. Well, assume that $MU\leq L\leq NU$ and $L\unlhd H$. Then $M\leq L\cap N\leq N$. If $N=L\cap N$ then $N\leq L$ which implies that $L=NU$. Otherwise, $L\cap N=M$ and hence $L=L\cap NU\leq(L\cap N)U=MU$.
		\item In general, this criterion is not necessary. For example, one can look on the sheaf $(\mathcal{G}=C_2\times T,T,\pi)$ with the natural projection $\pi:C_2\times T\to T$. Let $G=\prod_T^{\mathcal{C}} \mathcal{G}$, $B=C_2\times C_2$ and $A=C_2$. All homomorphisms are defined by extending the natural isomorphisms $C_2\to C_2$. In addition, let $S\subseteq T$ be a subset which is not closed and define $\{B_t:t\in T\}$ to be $\{C_2\times \{0\}:t\in S\}\cup \{\{0\}\times C_2:t\notin S\}$. One easily sees that $$(\varphi:G\to A,\alpha:B\to A, \{B_t:t\in T\})$$ is not solvable.
	\end{enumerate}
	
\end{rem}\label{equivalent condition for restricted free product}
If we consider the \textbf{restricted} free pro-$\mathcal{C}$ product, then the condition of Proposition \ref{condition for being free product} is in fact an equivalent condition, as can be shown by the following lemma:
\begin{lem}\label{restricted free product solve}
	Let $G$ be the restricted free pro-$\mathcal{C}$ product of its set of closed subgroups $\{G_i:i\in I\}$. Then $G$ can solve every finite embedding problem 
	\begin{equation}\label{embedding problem for restricted}
		\begin{gathered}
			(\varphi:G\to A,\alpha:B\to A,\{B_i:i\in I\})
		\end{gathered}
	\end{equation} in the category of pro-$\mathcal{C}$ groups.
\end{lem}
\begin{proof}
	Let \eqref{embedding problem for restricted} be a finite embedding problem for $(G,\{G_i\in I\})$ in the category of pro-$\mathcal{C}$ groups. Since $A$ is finite there is some finite subset $J\subseteq I$ such that for every $i\in I\setminus J$, $\varphi(G_i)=\{e\}$. In particular, for every $i\in I\setminus J$, $B_i=\{e\}$. Define for every $i\in I$ a map $\psi_i:G_i\to B$ by taking $\psi_i=\alpha_i^{-1}\circ \varphi|_{G_i}$. Here $\alpha_i^{-1}$ denotes the inverse of the isomorphism $\alpha_i=\alpha|_{B_i}:B_i\to \varphi(G_i)$. Then we have a converging to 1 map $D\to B$ whose restriction to every $G_i$ is a continuous homomorphism. By definition of the free pro-$\mathcal{C}$ product, there exists a continuous homomorphism $\psi:G\to B$ whose restriction to every $G_i,i\in I$ equals to $\psi_i$. In addition, since $\varphi$ and $\alpha\circ\psi$ are two continuous homomorphisms which identify on $D$ they identify on the abstract subgroup generated by $D$ and hence on $G$.
\end{proof}
Restricted free pro-$p$ products also satisfy the following property, which is a generalization of \cite[Exercise 9.1.22]{ribes2000profinite} for free pro-$p$ products of finitely many groups:
\begin{lem}\label{free product of conjugations}
	Let $G=\coprod_I^{r,p}G_i$ be a restricted free pro-$p$ product and let $g_i,i\in I$ be elements in $G$. Then $G=\coprod_I^{r,p}g_iG_ig_i^{-1}$.
\end{lem} 
\begin{proof}
	First we show that $\{g_iG_ig_i^{-1}:i\in I\}$ is a set of closed subgroups converging to 1. Let $U\leq_o G$. Choose some open normal subgroup $V$ of $G$ such that $V\leq U$. Since $\{G_i:i\in I\}$ is a set of closed subgroups converging to 1 then there exists a finite subset $J\subset I$ such that $G_i\leq V$ for all $i\in I\setminus J$. Since $V$ is normal, $g_iG_ig_i^{-1}\leq V$ for all $i\in I\setminus V$, so we are done. Now we show that $\bigcup_{i\in I}g_iG_ig_i^{-1}$ generates $G$. Recall that for every profinite group, a subset $X\subset G$ generates $G$ if and only if $X\Phi(G)$ generates $G/\Phi(G)$ where $\Phi(G)$ is the Frattini subgroup of $G$ (see \cite[Corollary 2.8.5]{ribes2000profinite}). Moreover, if $G$ is a pro-$p$ group then $G/\Phi(G)$ is an abelian group (see \cite[Lemma 2.8.7]{ribes2000profinite}). Hence, $(\bigcup_{i\in I}g_iG_ig_i^{-1})\Phi(G)/\Phi(G)=(\bigcup_{i\in I}G_i)\Phi(G)/\Phi(G)=G/\Phi(G)$ as $G$ is generated by $\bigcup_{i\in I}G_i$. So we are done. 

Now, let us look at the set of homomorphisms $\varphi_i:G_i\to G$ which are defined by $\varphi_i(g)=g_igg_i^{-1}$. Since $\{g_iG_ig_i^{-1}:i\in I\}$ is a set of closed subgroups converging to 1 and $G$ is the restricted free pro-$p$ product of $\{G_i\}_{i\in I}$, the set of homomorphisms $\{\varphi_i:G_i\to G\}$ induces a homomorphism $\varphi:G\to G$. Since $G$ is generated by $\bigcup_{i\in I}g_iG_ig_i^{-1}$ then $\varphi$ is in fact an epimorphism. We wish to construct a section $\psi:G\to G$. Since $\varphi$ is onto, for every $i\in I$ there exists some $x_i\in G_i$ such that $\varphi(x_i)=g_i$. Let $H_i=x_i^{-1}G_ix_i$. The set $\{H_i:i\in I\}$ is converging to 1, and $\bigcup_{i\in I}H_i$ generates $G$ by the same arguments. Define maps $\psi_i:G_i\to G$ by $\psi_i(g)=x_i^{-1}gx_i$. Then the set of homomorphisms $\{\psi_i:G_i\to G:i\in I\}$ is converging to 1 and thus induces a homomorphism $\psi:G\to G$. In fact since $\bigcup_{i\in I}H_i$ generates $G$, $\psi:G\to G$ is an epimorphism. Observe that for every $i\in I$ and $g\in G_i$, $\varphi\circ\psi(g)=g$. Thus $\varphi\circ \psi:G\to G$ is a homomorphism extends the set of homomorphisms $\operatorname{Id}_{G_i}:G_i\to G_i$. By the uniqueness in the universal property we get that $\varphi\circ \psi=Id_G$. Thus $\psi$ is injective and surjective which implies that $\psi$, and hence also $\varphi$, are isomorphisms. Since $\varphi$ maps $G_i$ isomorphically onto $g_iG_ig_i^{-1}$ one easily concludes that $G=\coprod_{i\in I}g_iG_ig_i^{-1}$.

\end{proof}
\begin{cor}\label{free product for every coseds}
In \cite[Corollary 7.3.2]{ribes2017profinite} it has been proven that for every open subgroup of a restricted free pro-$\mathcal{C}$ product $U\leq \coprod_I^{r,\mathcal{C}}G_i$ \textbf{there exist} sets $D_i$ of double coset representatives for $U\backslash G/G_i$ which contain 1, such that the family of inclusions $$\{U\cap xG_ix^{-1}|i\in I,x\in D_i\}$$ converges to 1 and $U$ is the restricted free pro-$\mathcal{C}$ product $$U=\left[\coprod_{i\in I,x\in D_i}^{r,\mathcal{C}}(U\cap xG_ix^{-1}) \right]\amalg^{\mathcal{C}} F$$ where $F$ is a free pro-$\mathcal{C}$ group. By Lemma \ref{free product of conjugations} we can conclude that if $\mathcal{C}$ is the variety of finite $p$-groups, then \textbf{for every} sets $D_i$ of double coset representatives for $U\backslash G/G_i$, the family of inclusions $$\{U\cap xG_ix^{-1}|i\in I,x\in D_i\}$$ converges to 1 and $U$ is the restricted free pro-$p$ product $$U=\left[\coprod_{i\in I,x\in D_i}^{r,p}(U\cap xG_ix^{-1}) \right]\amalg^{p} F$$ where $F$ is a free pro-$p$ group. Indeed, assume that $x,y$ are two representatives of the same double coset $UzG_i$. Then there exist $u\in U,g\in G_i$ such that $x=uyg$. Hence $U\cap xG_ix^{-1}=U\cap uygG_ig^{-1}y^{-1}u^{-1}=u(U\cap yG_iy^{-1})u^{-1}$, and by Lemma \ref{free product of conjugations} we are done.
\end{cor}
Now we are ready to prove the main theorem.
\begin{thm}[Main Theorem]
	Let $G=\coprod_{i\in I}^{r,p}G_i$ be a restricted free pro-$p$ product and $H\leq G$ a closed subgroup. Assume that:
\begin{enumerate}
	\item The space $H\backslash G/G_i$ is finite for every $i\in I$, and
	\item There exist sets $D_i$ of double coset representatives for $H\backslash G/G_i$ such that $H$ is generated by the subgroups $\{H\cap xG_ix^{-1}:i\in I,x\in D_i\}$.
\end{enumerate} 
Then the set $\{H\cap xG_ix^{-1}:i\in I,x\in D_i\}$ of subgroups of $H$ converges to 1 and  $H=\coprod_{i\in I,x\in D_i}^{r,p}H\cap xG_ix^{-1}$.
\end{thm}
\begin{proof}
	Let $G=\coprod_I^{r,p}G_i$ be a restricted free pro-$p$ product and $H\leq G$ be a subgroup as in the theorem. First we claim that the set ${\mathcal{X}=\{H\cap xG_ix^{-1}:i\in I,x\in D_i\}}$ converges to 1 in $H$. Indeed, let $U$ be an open normal subgroup of $H$, then there exists some $V\unlhd_o G$ such that $V\cap H\leq U$. Since $\{G_i:i\in I\}$ converges to 1 in $G$, there exists a finite subset $J\subset I$ such that  $G_i\leq V$ for all $i\in I\setminus J$. Hence for all $i\in I\setminus J, x\in D_i$, $xG_ix^{-1}\leq V$, and thus for all $i\in I\setminus J, x\in D_i$, $H\cap xG_ix^{-1}\leq H\cap  V\leq U$. Since $D_i$ is finite for every $i\in I$, we conclude that the set of all groups in $\mathcal{X}$ which are not contained in $U$ is finite.
	
	Now let $$(\varphi:H\to A,\alpha:B\to A,\{B_{i,x}:i\in I,x\in D_i\})$$ be a finite embedding problem for $(H,\mathcal{X})$ in the category of pro-$p$ groups. Since $\ker\varphi$ is open in $H$, there is some open normal subgroup $O\unlhd_oG$ such that $O\cap H\leq \ker\varphi$. By the convergence to 1 property of $\{G_i:i\in I\}$ together with the normality of $O$, there is a finite subset $J\subset I$ such that $xG_ix^{-1}\leq O$ for all $i\in I\setminus J, x\in D_i$. Let $i\in J$, $x\ne y$ in $D_i$. Then $x\notin HyG_i$. Hence there exists an open subgroup $H\leq U_{i,x,y}$ such that $x\notin U_{i,x,y}yG_i$. Taking the intersection $U=\bigcap_{i\in J,x\ne y\in D_i}U_{i,x,y}$ we get that for every $H\leq V\leq_o U$ and for all $i\in J$, $D_i$ is also a set of representatives for $V\backslash G/G_i$. Indeed, by the choice of $U$, for all $i\in J,x\ne y\in D_i$ $x\notin VyG_i$. Moreover, for all $g\in G$ and $i\in I$, $g\in HxG_i\leq VxG_i$ for some $x\in D_i$.     
	Put $L=HO$ and extend $\varphi$ to a homomorphism $\varphi':L\to H$ by letting $\varphi'(g)=e$ for every $g\in O$. In particular, for every $i\in I\setminus J$ and every $g\in G, g'\in G_i$ such that $gg'g^-1\in L$,  $\varphi'(gg'g^{-1})=\{e\}$. 
	
	Let $i\in J,x\in D_i$. Since $H\cap xG_ix^{-1}\leq (H\cap xG_ix^{-1})O$ there exists some open subgroup $H\leq V_{i,x}\leq L$ such that $V_{i,x}\cap xG_ix^{-1}\leq (H\cap xG_ix^{-1})O$. In particular, for every $H\leq V\leq V_{i,x}$, $$\varphi(H\cap xG_ix^{-1})=\varphi'(H\cap xG_ix^{-1})\leq \varphi'(V\cap xG_ix^{-1})\leq \varphi'(V_{i,x}\cap xG_ix^{-1})$$ $$\leq \varphi'(H\cap xG_ix^{-1}O)=\varphi'(H\cap xG_ix^{-1})$$ and hence $\varphi'(V\cap xG_ix^{-1})=\varphi(H\cap xG_ix^{-1})$.
	
	Let $H\leq K\leq \bigcap_{i\in J, x\in D_i} V_{i,x}$ be an open subgroup of $G$ containing $H$ such that  $D_i$ is a set of representatives for $K\backslash G/G_i$ for all $i\in J$. Such $K$ can be chosen by the previous paragraph.
	
	By Corollary \ref{free product for every coseds} there exist sets $E_i$ of double coset representatives for every $i\in I$, such that $E_i=D_i$ for every $i\in J$ and $K$ is the restricted free pro-$p$ product: $$K=[\coprod_{i\in I}^{r,p}\coprod_{x\in E_i}^{p}(K\cap xG_ix^{-1}) ]\amalg^p F$$ for some free pro-$p$ subgroup $F\leq K$.

	Consider the restriction of $\varphi'$ to $K$. Observe that $\varphi'(K\cap xG_ix^{-1})=\{e\} $ for every $i\in I\setminus J$ and $x\in E_i$. In addition, for every $i\in j$ and $x\in E_i=D_i$, $\varphi'(K\cap xG_ix^{-1})=\varphi(H\cap xG_ix^{-1})$. For every $i\in I, x\in E_i$, define a subgroup $C_{i,x}\leq B$ as follows: if $i\in I\setminus J$ then $C_{i,x}=\{e\}$ for every $x\in E_i$. Otherwise, $C_{i,x}=B_{i,x}$. We get that $(\varphi':K\to A,\alpha:B\to A,\{C_{i,x}:i\in I,x\in E_i\})$ is a finite embedding problem in the category of pro-$p$ groups for $$(K'=\coprod_{i=1}^{r,p}\coprod_{x\in E_i}^p(K\cap xG_ix^{-1}),\{K\cap xG_ix^{-1}:i\in I,x\in E_i\})$$ By Lemma \ref{restricted free product solve} there exists a homomorphism $\psi':K'\to B$ such that $\alpha\circ\psi'=\varphi'$ on $K'$, and for every $i\in I,x\in E_i$, $\psi'(K\cap xG_ix^{-1})=C_{i,x}$. Since $F$ is a free pro-$p$ group, there exists a homomorphism $\psi'':F\to B$ such that $\psi''\circ \alpha=\varphi'$ on $F$. By the uniquen extension property of free pro-$p$ products, there exists a homomorphism $\psi:K\to B$ extending $\psi'$ and $\psi''$ such that $\alpha\circ\psi=\varphi'$ on $K$. Let $\psi|_H$ be the restriction of $\psi$ to $H$. Then $\alpha\circ\psi|_H=\varphi$. Now let $i\in I$ and $x\in D_i$. First assume that $i\notin J$. Then for every $x\in D_i$ $\varphi(H\cap xG_ix^{-1})=\{e\}$. Recall that for every $g\in G, g'\in G_i$ $\varphi'(gg'g^{-1})=\{e\}$ whenever it is defined. Let $y\in E_i$ be such that $x\in KyG_i$. By definition, $C_{i,y}=\{e\}$ and $\psi'(K\cap yG_iy^{-1})=\{e\}$. Let $k\in K,g\in G_i$ be such that $x=kyg$. Then $$H\cap xG_ix^{-1}\leq K\cap xG_ix^{-1}=K\cap kygG_ig^{-1}y^{-1}k^{-1}=k(K\cap yG_iy^{-1})k^{-1}$$. Since $$\psi(k(K\cap yG_iy^{-1})k^{-1})=\psi(k)\psi(K\cap yG_iy^{-1})\psi(k)^{-1}=\{e\}$$ we conclude that $\psi|_H(H\cap xG_ix^{-1})=\{e\}$ as required.
	
	Now assume that $i\in J$. Let $x\in D_i$. Then $x\in E_i$ and $\psi(K\cap xD_ix^{-1})=C_{i,x}=B_{i,x}$. Hence, $\psi|_H(H\cap xG_ix^{-1})\leq B_{i,x}$. However, since $\alpha$ maps $B_{i,x}$ isomorphically onto $\varphi(H\cap xG_ix^{-1})$ we conclude that  $\psi|_H(H\cap xG_ix^{-1})= B_{i,x}$. In conclusion, by Proposition \ref{condition for being free product} we are done. 
\end{proof}

	\bibliographystyle{plain}
	
\end{document}